\documentclass[12pt,twoside]{amsart}
\usepackage{amssymb}
\usepackage{verbatim}
\usepackage{amsmath}
\usepackage{bm}
\usepackage{a4wide}
\usepackage[latin1]{inputenc}
\usepackage[T1]{fontenc}
\usepackage{times}
\usepackage{amssymb,latexsym}
\usepackage{enumerate}
\usepackage{mathrsfs}

\makeatletter
\newcommand{\sumprime}{\if@display\sideset{}{'}\sum%
            \else\sum'\fi}
\makeatother

\begin{document}

\numberwithin{equation}{section}

\newtheorem{theorem}{Theorem}[section]
\newtheorem{proposition}[theorem]{Proposition}
\newtheorem{conjecture}[theorem]{Conjecture}
\def\theconjecture{\unskip}
\newtheorem{corollary}[theorem]{Corollary}
\newtheorem{lemma}[theorem]{Lemma}
\newtheorem{observation}[theorem]{Observation}
\newtheorem{definition}{Definition}
\numberwithin{definition}{section} 
\newtheorem{remark}{Remark}
\def\theremark{\unskip}
\newtheorem{kl}{Key Lemma}
\def\thekl{\unskip}
\newtheorem{question}{Question}
\def\thequestion{\unskip}
\newtheorem{example}{Example}
\def\theexample{\unskip}
\newtheorem{problem}{Problem}

\thanks{Supported by National Natural Science Foundation of China, No. 12271101}

\address{Department of Mathematical Sciences, Fudan University, Shanghai, 20043, China}

\email{boychen@fudan.edu.cn}

\title[Capacities,  Green function and Bergman functions]{Capacities, Green function and Bergman functions}

\author{Bo-Yong Chen}

\begin{abstract}
Using the logarithmic capacity,  we give quantitative estimates of the Green function,  as well as lower bounds of the Bergman kernel for bounded pseudoconvex domains in $\mathbb C^n$ and the Bergman distance for bounded planar domains.  In particular,  it is shown that the Bergman kernel satisfies $K_\Omega(z)\gtrsim \delta_\Omega(z)^{-2}$ for any bounded pseudoconvex domain with $C^0-$boundary.  An application to holomorphic motions is given.  
\end{abstract}

\maketitle

\section{Introduction}
 Let $\Omega$ be a bounded domain in $\mathbb C$ with $0\in \partial \Omega$. 
 The famous Wiener's criterion states that $0$  is a regular point if and only if  
 $$
 \sum_{k=1}^\infty \frac{k}{\log [1/\mathcal C_l(  A_k\cap \Omega^c )]}=\infty,
 $$
  where $
    A_k:=\{z\in \mathbb C:\frac1{2^{k+1}}\le |z|\le \frac1{2^k}\}.
    $ and $\mathcal C_l(\cdot)$ stands for the standard logarithmic capacity (relative to $\infty$).  
  It is also known that the regularity of $0$ implies that the Bergman kernel is exhaustive at $0$ (cf. \cite{Ohsawa}) and the Bergman metric is complete at $0$ (cf. \cite{BP}, \cite{Herbort}). Zwonek \cite{Zwonek} showed that the Bergman kernel  is exhaustive at $0$ if and only if 
  $$
\sum_{k=1}^\infty \frac{2^{2k}}{\log [1/\mathcal C_l (A_k(z)\cap \Omega^c)]}\rightarrow \infty\ \ \ (z\rightarrow 0),
 $$
 where $A_k(z):=\{w\in \mathbb C:\frac1{2^{k+1}}\le |w-z|\le \frac1{2^k}\}$. 
 On the other hand, a similar characterization for the Bergman completeness at $0$ is still missing, although some partial results exist (cf. \cite{PZ},  \cite{Jucha}).  Only recently,  interesting lower and upper bounds of the Bergman kernel  $K_\Omega(z)$ for planar domains were obtained in terms of the logarithmic capacity of $\Omega^c$ relative to $z$ (cf.  \cite{BlockiSuita},  \cite{BZ}). 
 
In the spirit of Wiener's criterion,  we give the following 
 
 \begin{theorem}\label{th:CapBerg}
 Let $\Omega$ be a bounded pseudoconvex domain in $\mathbb C^n$.  Suppose there are constants $\alpha>1$ and $\varepsilon>0$ such that for any $z\in \Omega$,  there exists a complex line $\mathcal L_z\ni z$ such that
 \begin{equation}\label{eq:CB_1}
\mathcal C_l\left(\mathcal L_z\cap B(z,\alpha \delta_\Omega(z))\cap \Omega^c\right)\ge \varepsilon\, \delta_\Omega(z),
 \end{equation}
 where $\delta_\Omega(z):=d(z,\partial \Omega)$ and $B(z,r)=\{w\in \mathbb C^n:|w-z|<r\}$.   Then 
 \begin{equation}\label{eq:CB_2}
 K_\Omega(z) \gtrsim  \delta_\Omega(z)^{-2},\ \ \ z\in \Omega.
 \end{equation}
 \end{theorem}
 
 It is well-known that \eqref{eq:CB_2} holds for any bounded pseudoconvex domain with Lipschitz boundary (cf.  \cite{OT}).  As a consequence of Theorem \ref{th:CapBerg},  we obtain    the following
 
 \begin{theorem}\label{th:KernelEstimate}
 Let $\Omega$ be a bounded pseudoconvex domain in $\mathbb C^n$.  Then  \eqref{eq:CB_2} holds under one of the following conditions: 
\begin{enumerate} 
\item[$(a)$] $\partial\Omega$ is $C^0$,  i.e.,  it can be written locally as the graph of a continuous function.  
\item[$(b)$] There are constants $\varepsilon,r_0>0$  such that
     \begin{equation}\label{eq:CB_3}
    |B(\zeta,r)\cap \Omega^c|\ge \varepsilon |B(\zeta,r)|,\ \ \   \zeta\in \partial \Omega,\,0<r\le r_0,
     \end{equation}
      where $|\cdot|$ stands for the volume.       
\item[$(c)$] $\Omega$ is fat $($i.e.,  $\overline{\Omega}^\circ=\Omega)$ and\/ $\overline{\Omega}$ is $\mathcal O(U)-$convex,  where $U$ is a neighborhood of\/  $\overline{\Omega}$.  
\item[$(d)$] $\Omega$ is strongly hyperconvex.  
      \end{enumerate}
 \end{theorem}
 
Here are a few remarks.   Recall that the $\mathcal O(U)-$convex hull of a compact set $E\subset U$ is given by   
      $$
      \widehat{E}_{\mathcal O(U)}:=\left\{z\in U: |f(z)|\le \sup_E |f| \ \text{for all}\ f\in \mathcal O(U)\right\}.
      $$   
    We say that  $E$ is $\mathcal O(U)-$convex if $E= \widehat{E}_{\mathcal O(U)}$;  in particular,  if $U=\mathbb C^n$ then $E$ is called polynomially convex.  
 Recall also that a bounded domain $\Omega\subset \mathbb C^n$ is called strongly hyperconvex if there exists a continuous negative psh function $\rho$ defined in a neighborhood $U$ of $\overline{\Omega}$ such that $\Omega=\{z\in U:\rho(z)<0\}$.   Based on the work of Zwonek \cite{Zwonek},  Pflug-Zwonek \cite{PZ}  proved that for any $\varepsilon>0$ there exists a constant $C_\varepsilon>0$ such that $K_\Omega(z)\ge C_\varepsilon\, \delta_\Omega(z)^{-2+\varepsilon}$.  
Note that conditions of type \eqref{eq:CB_3} are popular in second order elliptic boundary problems $($cf.  \cite{Kenig},  Chapter 1,\,\S\,1$)$,  and spectrum theory of the Laplacian $($cf.  \cite{Lieb},\,\cite{Davis}$)$.  
 
 We also have the following analogous result which  does not follow from Theorem \ref{th:CapBerg}.
  
  \begin{proposition}\label{prop:Runge}
  If\/ $\Omega\subset \mathbb C^n$ is a bounded pseudoconvex Runge domain,  then \eqref{eq:CB_2} holds.
  \end{proposition}
  
  \begin{remark}
  It is known that every star domain in $\mathbb C^n$ is a Runge domain $($cf.  \cite{Almer} or \cite{Kasimi}$)$.
  \end{remark}
  
 The Bergman kernel estimate \eqref{eq:CB_2} has an unexpected application to holomorphic motions.  Following \cite{MSS},  we define  a holomorphic motion of a set $E\subset \mathbb C$ 
as a mapping
$
F:\mathbb D \times E \rightarrow {\mathbb  C},
$
which enjoys the following properties: 
\begin{enumerate}
\item $F(0,z)=z$ for all $z\in
E$.
\item  For every fixed $\lambda\in \mathbb D$, the mapping
$F(\lambda,\cdot):E\rightarrow {\mathbb  C}$ is an injection.
\item  For
every fixed $z\in E$, the mapping $F(\cdot,z):\mathbb D\rightarrow {\mathbb 
C}$ is holomorphic. 
\end{enumerate}

Let $F:\mathbb D\times {\mathbb C}\rightarrow
{\mathbb C}$ be a holomorphic motion of\/  $\mathbb C$\/ fixing $\infty$.  Define quasi-bidiscs as follows
   $$
   \Gamma_r(F):=\left\{(\lambda,F(\lambda,z)): \lambda\in \mathbb D_r,z\in \mathbb D\right\},\ \ \ 0<r<1,
   $$
   where $\mathbb D_r=\{z\in \mathbb C:|z|<r\}$.   
 We obtain a geometric property of quasi-bidiscs as follows.    

\begin{theorem}\label{th:HM}
Let $\delta_\lambda(w)$ denote the boundary distance of the slice $F(\lambda,\mathbb D)$,  $\lambda\in \mathbb D$.  Then for any $0<r'<r<1$,  there exists a constant $C>0$ such that
\begin{equation}\label{eq:HM_1}
\delta_\lambda(w)
\le C \delta_{ \Gamma_r(F)}(\lambda,w) 
\end{equation}
whenever $|\lambda|\le r'$.  
\end{theorem}

\begin{remark}
The advantage of Theorem \ref{th:HM} is that the analysis in terms of $\delta_\lambda$ is easier on $\Gamma_r(F)$  $($compare \cite{ChenH-index},  p.  1450--1451$)$. 
\end{remark}

The argument in Zwonek \cite{Zwonek}  yields  the following result closely related to Theorem \ref{th:CapBerg}.

 \begin{theorem}\label{th:CapBerg_2}
 Let $\Omega$ be a bounded pseudoconvex domain in $\mathbb C^n$.  Suppose there are constants $\alpha>1$,  $\beta\ge 1$ and $\varepsilon>0$ such that for any $z\in \Omega$,  there exists a complex line $\mathcal L_z\ni z$ such that
 \begin{equation}\label{eq:CB3}
\mathcal C_l\left(\mathcal L_z\cap B(z,\alpha \delta_\Omega(z))\cap \Omega^c\right)\ge \varepsilon\, \delta_\Omega(z)^\beta.
 \end{equation}
  Then 
 \begin{equation}\label{eq:CB4}
 K_\Omega(z) \gtrsim  \delta_\Omega(z)^{-2}|\log \delta_\Omega(z)|^{-1}
 \end{equation}
  for all $z$ sufficiently close to $\partial \Omega$.
 \end{theorem}
 
 \begin{corollary}\label{cor:VolBerg_2}
 Let $\Omega$ be a bounded pseudoconvex domain in $\mathbb C^n$.  Suppose there are constants $\varepsilon,r_0>0$ and $\beta> 1$,   such that
      \begin{equation}\label{eq:CB5}
    |B(\zeta,r)\cap \Omega^c|\ge \varepsilon |B(\zeta,r)|^\beta,\ \ \   \zeta\in \partial \Omega,\,0<r\le r_0,
      \end{equation} 
      then \eqref{eq:CB4}  holds.  
 \end{corollary}

For the Bergman kernel,  the cerebrated Ohsawa-Takegoshi extension theorem \cite{OT} serves as a bridge passing from one-dimensional estimates to high dimensions.  Unfortunately,  such a powerful tool is not available for estimating other objects like the pluricomplex Green function or the Bergman metric/distance.  Thus we have to focus on bounded  domains $\Omega\subset \mathbb C$ in the sequel.     
    
Motivated  by the work of Carleson-Totik \cite{CarlesonTotik},  we give the following
 
   \begin{definition}
           Let $\varepsilon>0$ and $0<\lambda< 1$ be fixed. For every $a\in \partial \Omega$  we set 
                      \begin{eqnarray*}
           K_t(a) & := & \overline{\mathbb D_{t}(a)} - \Omega;\ \ \ \mathbb D_{t}(a):=\{z:|z-a|<t\} \\
                     \mathcal N_a(\varepsilon,\lambda) & := & \left\{n\in \mathbb Z^+: \mathcal C_l\left(K_{\lambda^n}(a)\right)\ge \varepsilon \lambda^{ n}\right\}\\
           \mathcal N_a^n(\varepsilon,\lambda) &:=&\mathcal N_a(\varepsilon,\lambda)\cap \{1,2,\cdots,n\}.
           \end{eqnarray*}
           We define the $(\varepsilon,\lambda)-$capacity density of $\partial \Omega$ at $a$ by
           $$
           \mathcal D_a(\varepsilon,\lambda):=\liminf_{n\rightarrow\infty} \frac{|\mathcal N_a^n(\varepsilon,\lambda)|}n.
           $$
           We define the weak and strong $(\varepsilon,\lambda)-$capacity density of $\partial \Omega$ by        
           $$
           \mathcal D_W(\varepsilon,\lambda)  :=  \liminf_{n\rightarrow\infty} \frac{\inf_{a\in \partial \Omega}|\mathcal N_a^n(\varepsilon,\lambda)|}n      
           $$
           and 
           \begin{eqnarray*}           
                      \mathcal D_S (\varepsilon,\lambda)  :=   \liminf_{n\rightarrow\infty} \frac{|\bigcap_{a\in \partial \Omega}\mathcal N_a^n(\varepsilon,\lambda)|}n
         \end{eqnarray*}
         respectively.
               \end{definition}                
                    
       It is easy to see  that $ \mathcal D_W (\varepsilon,\lambda)\ge  \mathcal D_S (\varepsilon,\lambda)$, and
        $\mathcal C_l(K_{t}(a))\ge \varepsilon t$,  $\forall\,t>0$,   implies $ \mathcal D_a(\varepsilon,\lambda)=1$. Recall that $\partial \Omega$ is said to be\/ {\it uniformly perfect} if $\inf_{a\in \partial \Omega} \mathcal C_l(K_{t}(a))\ge \varepsilon t$,  $\forall\,t>0$ (cf. \cite{Pommerenke}). Thus if $\partial \Omega$ is uniformly perfect then $\mathcal D_W (\varepsilon,\lambda)=\mathcal D_S (\varepsilon,\lambda)=1$ for some $\varepsilon>0$. 
 On the other hand, it was pointed out in \cite{CarlesonTotik} that the domain   
  $$
  \Omega= \mathbb D-\{0\}-\bigcup _{k=1}^\infty \left[2^{-2^{2k+1}},2^{-2^{2k}}\right]  
  $$  
  satisfies $D_W (\varepsilon,1/2)>0$ for some $\varepsilon>0$ while
   $\partial \Omega$ is non-uniformly perfect. Actually, one may verify that $D_S (\varepsilon,1/2)>0$.
   
  Let  $z_0$ be a fixed point in $\Omega$.   Carleson-Totik proved the following  quantitative analogue of the sufficient part of Wiener's criterion.  
   
   \begin{theorem}[cf. \cite{CarlesonTotik}]\label{th:CT}
   If\/ $ \mathcal D_W(\varepsilon,\lambda)>0$\/ for some $\varepsilon,\lambda$, then there exists $\beta>0$ such that the Green function $g_\Omega$ satisfies
           \begin{equation}\label{eq:CT}
           -g_\Omega(z,z_0)\lesssim \delta_\Omega(z)^\beta
           \end{equation}
           for all $z$ sufficiently close to $\partial \Omega$.
  \end{theorem}   
  
   \begin{remark}
        It is remarkable that  the converse of Theorem \ref{th:CT} holds under the additional condition that  $\Omega$ contains a fixed size cone with vertex at any $a\in \partial\Omega$
        $($cf. \cite{CarlesonTotik}$)$.
        \end{remark}

  Since the original proof in \cite{CarlesonTotik} is rather technical,  it seems worthwhile to give a more transparent approach.  Actually,  Theorem \ref{th:CT} turns out to be a simple consequence of a quantitative estimate of the capacity potential given in  Theorem \ref{th:basic},  which can be applied in more general context. 
   
  \begin{definition}
           For $\varepsilon>0$, $0<\lambda< 1$ and $\gamma>1$  we set 
           \begin{eqnarray*}
           \mathcal N_a(\varepsilon,\lambda,\gamma) &:=& \left\{n\in \mathbb Z^+: \mathcal C_l(K_{\lambda^{n}}(a))\ge \varepsilon \lambda^{\gamma n}\right\}\\
        \mathcal N_a^n (\varepsilon,\lambda,\gamma)   &:=& \mathcal N_a(\varepsilon,\lambda,\gamma)\cap  \{1,2,\cdots,n\}.          \end{eqnarray*}
           We define the $(\varepsilon,\lambda,\gamma)-$capacity density of $\partial \Omega$ at $a$ by
           $$
           \mathcal D_a(\varepsilon,\lambda,\gamma):=\liminf_{n\rightarrow\infty} \frac{\sum_{k\in \mathcal N_a^n(\varepsilon,\lambda,\gamma)} k^{-1}}{\log n}
           $$
           and the weak and strong $(\varepsilon,\lambda,\gamma)-$capacity densities of $\partial \Omega$ by
           $$
           \mathcal D_W(\varepsilon,\lambda,\gamma):=\liminf_{n\rightarrow\infty} \frac{\inf_{a\in \partial \Omega} \sum_{k\in \mathcal N_a^n(\varepsilon,\lambda,\gamma)} k^{-1}  }{\log n}
           $$
           and
           $$
            \mathcal D_S(\varepsilon,\lambda,\gamma):=\liminf_{n\rightarrow\infty} \frac{ \sum_{k\in \bigcap _{a\in \partial \Omega}\mathcal N_a^n(\varepsilon,\lambda,\gamma)} k^{-1} }{\log n}
          $$ 
          respectively.
                                \end{definition}
           
     Note that $ \mathcal D_W(\varepsilon,\lambda,\gamma)\ge \mathcal D_S(\varepsilon,\lambda,\gamma)$. If  $\mathcal C_l(K_t(a))\ge \varepsilon t^\gamma$,  $\forall\,t>0$,   for some $\varepsilon>0$ and $\gamma>1$, then
           $
            \mathcal D_a(\varepsilon,\lambda,\gamma)=1
                       $ 
                       in view of the following well-known formula
                       $$
                      \lim_{n\rightarrow \infty} \left( 1+\frac12+\cdots+\frac1n - \log n \right) = \text{Euler constant}.
                       $$   
                       
                       \begin{theorem}\label{th:logGreen}
                    \begin{enumerate}   
             \item[$(1)$]   If\/ $ \mathcal D_W(\varepsilon,\lambda,\gamma)>0$\/ for some $\varepsilon,\lambda,\gamma$, then there exists $\beta>0$ such that 
           \begin{equation}\label{eq:LG_1}
           -g_\Omega (z,z_0) \lesssim (-\log\delta_\Omega(z))^{-\beta}
           \end{equation}
           for all $z$ sufficiently close to $\partial \Omega$.
        \item[$(2)$]     
                    Suppose\/   $\inf_{a\in \partial \Omega} \mathcal C_l(K_t(a))\ge \varepsilon t^\gamma$ for some $\varepsilon>0$ and $\gamma>1$. For every $\tau<\frac1{\gamma-1}$ there exists $C>0$ such that  
           \begin{equation}\label{eq:LG_2}
            -g_\Omega (z,z_0) \le C (-\log \delta_\Omega(z))^{-\tau}
           \end{equation}
            for all $z$ sufficiently close to $\partial \Omega$.
            \end{enumerate}
                       \end{theorem}
                       
                       We also obtain the following lower bounds for the Bergman distance $d_B$.
    
     \begin{theorem}\label{th:LowDist}
     \begin{enumerate}
 \item[$(1)$] If\/ $ \mathcal D_S(\varepsilon,\lambda)>0$\/ for some $\varepsilon,\lambda$, then 
        \begin{equation}\label{eq:1.1}
         d_B(z_0,z) \gtrsim   |\log \delta_\Omega(z)|
       \end{equation}  
       for all $z$ sufficiently close to $\partial \Omega$.
 \item[$(2)$]  If\/ $ \mathcal D_W(\varepsilon,\lambda)>0$\/ for some $\varepsilon,\lambda$, then 
        \begin{equation}\label{eq:1.2}
         d_B(z_0,z) \gtrsim  \frac{|\log \delta_\Omega(z)|}{\log|\log \delta_\Omega(z)|}
       \end{equation}    
            for all $z$ sufficiently close to $\partial \Omega$.
            \item[$(3)$] If\/ $ \mathcal D_W(\varepsilon,\lambda,\gamma)>0$\/ for some $\varepsilon,\lambda,\gamma$, then 
        \begin{equation}\label{eq:1.3}
         d_B(z_0,z) \gtrsim  \log \log |\log \delta_\Omega(z)|
        \end{equation}   
        for all $z$ sufficiently close to $\partial \Omega$.
            \end{enumerate}
           \end{theorem}
           
           \begin{remark}
         $(1)$   In \cite{ChenEssay},  \eqref{eq:1.1} was verified  by a different method in case $\partial \Omega$ is uniformly perfect. 
         
         $(2)$ Estimate of type \eqref{eq:1.2} was first obtained by B\l ocki \cite{BlockiGreen} for bounded pseudoconvex domains with Lipschitz boundaries in $\mathbb C^n$ $($see also \cite{ChenH-index} and \cite{DO} for related results$)$. 
         
         $(3)$  We conjecture that \eqref{eq:1.3} can be improved to $ d_B(z_0,z) \gtrsim   \log |\log \delta_\Omega(z)|$.
            \end{remark}
           
    By now it becomes a standard method for lower bounds of Bergman functions by estimating the set 
    $
    \{g_\Omega(\cdot,z)\le -c\}
    $      
    for $c>0$,  where $g_\Omega(\cdot,z)$ stands for the (pluricomplex) Green function with pole at $z$.  Namely,  one has
    \begin{equation}\label{eq:BG_1}
    K_\Omega(z) \gtrsim_c | \{g_\Omega(\cdot,z)\le -c\}|^{-1},
    \end{equation}
     \begin{equation}\label{eq:BG_2}
    d_B(z,z')\gtrsim_c 1\ \ \text{whenever}\ \ \{g_\Omega(\cdot,z)\le -c\}\cap \{g_\Omega(\cdot,z')\le -c\}=\emptyset,
    \end{equation}
    where $A\gtrsim_c B$ means $A\ge CB$ for some constant $C=C(c)>0$ (cf.  \cite{BlockiGreen}, \cite{BP}, \cite{BZ},  \cite{Chen99,CZ,ChenEssay,ChenH-index},  \cite{DO}, \cite{Herbort},  \cite{Ohsawa}).  In the proofs of Theorem \ref{th:CapBerg} and Theorem \ref{th:LowDist},  we use  \eqref{eq:BG_1} and \eqref{eq:BG_2} respectively.

    \section{Capacities}
In this section we shall review different notions of  capacities and present some basic properties of them.
Let $\Omega$ be a bounded domain  in $\mathbb C$ and $K\subset \Omega$  a compact (non-polor) set in $\Omega$.  We define the Dirichlet capacity $\mathcal C_d(K,\Omega)$ of $K$ relative to $\Omega$ by
\begin{equation}\label{eq:C_1}
\mathcal C_d(K,\Omega) = \inf_{\phi\in \mathcal L(K,\Omega)} \int_\Omega |\nabla \phi|^2 
\end{equation}                     
where $\mathcal L(K,\Omega)$ is the set of all locally Lipschitz functions $\phi$ on $\Omega$ with a compact support in $\overline{\Omega}$
 such that $0\le \phi\le 1$ and $\phi|_K=1$. If $\Omega = \mathbb C$,  then we write $\mathcal C_d(K)$ for $\mathcal C_d(K,\Omega)$.
 By the definition we have
   \begin{equation}\label{eq:C_2}
  K_1\subseteq K_2\ \ \text{and}\ \ \Omega_1\supseteq \Omega_2\ \  \Rightarrow\ \  \mathcal C_d(K_1,\Omega_1)\le \mathcal C_d(K_2,\Omega_2).
 \end{equation}
     In view of Dirichlet's principle, the infimum in \eqref{eq:C_1} is attained  at the function $\phi_{\min}$ which is exactly the Perron solution to the following (generalized) Dirichlet problem in $\Omega\backslash K$:
    \begin{equation}\label{eq:C_3}
    \Delta u =0;\ \ u =0\  \text{n.e.\ on\ } \partial \Omega;\ \ u=1\ \text{n.e.\ on\ } \partial K.
    \end{equation}
    We call $\phi_{\min}$ the\/ {\it capacity potential}\/ of $K$ relative to $\Omega$.  In case $\partial \Omega$ and $\partial K$ are both $C^2-$smooth, integration by parts gives  
    \begin{eqnarray}\label{eq:C_4}
    \mathcal C_d(K,\Omega) & = & \int_\Omega |\nabla \phi_{\min}|^2 = \int_{\Omega\backslash K} |\nabla \phi_{\min}|^2\nonumber\\
    & = & -\int_{\Omega\backslash K} \phi_{\min}\,\Delta \phi_{\min} +\int_{\partial (\Omega\backslash K)} \phi_{\min} \frac{\partial \phi_{\min}}{\partial \nu} d\sigma\nonumber\\
    & = &  \int_{\partial  K} \frac{\partial \phi_{\min}}{\partial \nu} d\sigma=: - {\text{flux}}_{\partial K}\, \phi_{\min}
       \end{eqnarray}
    where $\nu$ is the outward unit normal vector fields on $\partial (\Omega\backslash K)$. By the Hopf lemma we conclude that $\partial \phi_{\min}/\partial \nu\ge 0$ holds on $\partial K$.
    
   Let $g_\Omega(z,w)$ be the (negative) Green function on $\Omega$. Let $z\in \Omega\backslash K$ be given. Since $\Delta g_\Omega(\cdot,z) =2\pi \delta_z$, where $\delta_z$ stands for  the Dirac measure at $z$, we infer from Green's formula that
    \begin{eqnarray}\label{eq:C_4'}
    2\pi \phi_{\min}(z) & = & \int_{\Omega\backslash K} \phi_{\min}\,\Delta g_\Omega(\cdot,z)= \int_{\Omega\backslash K} g_\Omega(\cdot,z)\,\Delta \phi_{\min}\nonumber\\
    && + \int_{\partial (\Omega\backslash K)} \phi_{\min}\,\frac{\partial g_\Omega(\cdot,z)}{\partial \nu} d\sigma
    - \int_{\partial (\Omega\backslash K)} g_\Omega(\cdot,z)\,\frac{\partial \phi_{\min}}{\partial \nu} d\sigma\nonumber\\
    & = & \int_{\partial K}   \frac{\partial g_\Omega(\cdot,z)}{\partial \nu} d\sigma - \int_{\partial  K} g_\Omega(\cdot,z)\,\frac{\partial \phi_{\min}}{\partial \nu} d\sigma\nonumber\\
    & = &   - \int_{\partial  K} g_\Omega(\cdot,z)\,\frac{\partial \phi_{\min}}{\partial \nu} d\sigma
       \end{eqnarray}
    because $g_\Omega(\cdot,z)$ is harmonic on $K$.  This equality combined with \eqref{eq:C_4} gives  the following fundamental inequality which connects the capacity, Green's function and the capacity potential: 
    \begin{equation}\label{eq:C_4''}
  \frac{\mathcal C_d(K,\Omega)}{2\pi}\, \inf_{\partial K}(-g_\Omega(\cdot,z))\le  \phi_{\min}(z)\le \frac{\mathcal C_d(K,\Omega)}{2\pi}\,\sup_{\partial K}(-g_\Omega(\cdot,z)),\ \ z\in \Omega\backslash K.
        \end{equation}
        Since $\Omega\backslash K$ can be exhausted by bounded domains with smooth boundaries, we conclude by passing to  a standard limit process that the same inequality holds for every compact set $K$.
        
    For a finite Borel measure $\mu$ on $\mathbb C$ whose support is contained in  $K$  we define its Green potential relative to $\Omega$ by 
    $$
    p_{\mu}(z) = \int_\Omega g_\Omega(z,w) d\mu(w), \ \ z\in \Omega. 
    $$
    Clearly, $p_\mu$ is negative, subharmonic on $\Omega$,  harmonic on $\Omega\backslash K$, and satisfies $p_\mu(z)=0$ n.e.  on  $ \partial \Omega$.  Given $\phi\in C^\infty_0(\Omega)$ we have
    \begin{eqnarray*}
    \int_\Omega p_\mu \Delta \phi dV & = & \int_\Omega \left[\int_\Omega g_\Omega(z,w) d\mu(w) \right] \Delta\phi(z)dV(z)\\
    & = & \int_\Omega \left[\int_\Omega g_\Omega(z,w) \Delta\phi(z)dV(z) \right]   d\mu(w) \ \ (\text{Fubini's theorem})\\
     & =  &  \int_\Omega \left[\int_\Omega \Delta g_\Omega(z,w) \phi(z)dV(z) \right]   d\mu(w)\ \ (\text{Green's formula})\\
     & = & \int_\Omega 2\pi \phi(w) d\mu(w).
     \end{eqnarray*}
     Thus we obtain
     \begin{equation}\label{eq:C_5}
     \Delta p_\mu =2\pi \mu
     \end{equation}
     in the sense of distributions. 
      The Green energy $I(\mu)$ of $\mu$ is given by
    $$
    I(\mu):= \int_\Omega p_\mu d\mu = \int_\Omega \int_\Omega g_\Omega(z,w) d\mu(z)d\mu(w). 
    $$
     By \eqref{eq:C_5} we have
     \begin{equation}\label{eq:C_6}
     I(\mu) = \frac1{2\pi} \int_\Omega p_\mu\, \Delta p_\mu=-\frac1{2\pi} \int_\Omega |\nabla p_\mu|^2.
     \end{equation}   
        Every compact set $K$ has an\/ {\it equilibrium measure} $\mu_{\max}$, which maximizes $I(\mu)$ among all Borel probability measures $\mu$ on $K$. A fundamental  theorem of Frostman states that  
        \begin{enumerate}
        \item[$(1)$] $p_{\mu_{\max}}\ge I(\mu_{\max})$ on $\Omega$;
     \item[$(2)$] $p_{\mu_{\max}}=I(\mu_{\max})$ on $K\backslash E$ for some $F_\sigma$ polor set $E\subset \partial K$. 
     \end{enumerate}
     
   By the uniqueness of the solution of the (generalized) Dirichlet problem we have
   \begin{equation}\label{eq:C_7}
   \phi_{\min}= p_{\mu_{\max}}/ I(\mu_{\max}).
     \end{equation}
    We define the Green capacity $\mathcal C_{g}(K,\Omega)$ of $K$ relative to $\Omega$ by
    $$
   \mathcal  C_{g}(K,\Omega):= e^{I(\mu_{\max})}.
    $$  
    It follows from \eqref{eq:C_6} and \eqref{eq:C_7} that 
    \begin{equation}\label{eq:C_8}
    \frac{\mathcal C_d(K,\Omega)}{2\pi}=-\frac1{\log \mathcal C_g(K,\Omega)}.
    \end{equation}
    
    Analogously, we may define the logarithmic capacity $\mathcal C_l(K)$ of $K$ by    
    $$
    \log \mathcal C_l(K):=\sup_\mu \int_\mathbb C \int_\mathbb C \log |z-w| d\mu(z) d\mu(w)
    $$
    where the supremum is taken over all Borel propability measures $\mu$ on $\mathbb C$ whose support is contained in $K$. Let $R$ be the diameter of $\Omega$ and set $d=d(K,\Omega)$. Since
    $$
    \log |z-w|/R \le g_\Omega(z,w)\le \log |z-w|/d,\ \ z,w\in K,
    $$
    we have
    \begin{equation}\label{eq:C_9}
   \log \mathcal C_l(K)-\log R \le \log \mathcal C_g(K,\Omega)\le \log \mathcal C_l(K)-\log d.
    \end{equation}
    
    \section{Estimates of the capacity potential}
    
   We first give a basic lemma as follows. 
    
    \begin{lemma}\label{lm:basic}
 Let $\Omega$ be a bounded domain in $\mathbb C$ with $0\in \partial \Omega$.   Let $0\le h\le 1$ be a harmonic function on $\Omega$ such that $h=0$ n.e. on $\partial \Omega \cap \mathbb D_{r_0}$ for some $r_0<1$.  For all $0<\alpha< 1/16$ and $r\le \alpha r_0$ we have 
    \begin{equation}\label{eq:G_0}
    \sup_{\Omega\cap \mathbb D_r} h \le \exp\left[-\frac{\log 1/(16\alpha)}{\log 1/\alpha} \int_{r}^{\alpha r_0} \left({t\log \frac{ t/\alpha}{2 \mathcal C_l(K_{ t})}} \right)^{-1} dt  \right]
    \end{equation}
    where $K_t:=\overline{\mathbb D}_t-\Omega$.
      \end{lemma}
    
    \begin{proof}
    The idea of the proof comes from \cite{GW} (see also \cite{ChenHolder}). Let $\mathbb D$ be the unit disc.
     For $t<r_0$ and $|z|=t$ we have
    \begin{equation}\label{eq:G_1}
    \sup_{\partial K_{\alpha t}}(-g_{\mathbb D}(\cdot,z))\le \log 2+  \sup_{\partial K_{\alpha t}}(-\log|\cdot-z|) \le \log 2-\log |t-\alpha t|\le \log 4/t.
       \end{equation}
  Let $\phi_{\alpha t}$ be the capacity potential of $K_{\alpha t}$ relative to $\mathbb D$.  By \eqref{eq:C_4''} and \eqref{eq:G_1} we have
  $$
  \phi_{\alpha t}(z)\le (\log 4/t)\cdot\frac{\mathcal C_d(K_{\alpha t},\mathbb D)}{2\pi}\ \ \ \ \text{for}\ \ |z|=t.
  $$
  It follows that for $z\in \Omega\cap \partial \mathbb D_t$, 
  \begin{equation}\label{eq:G_2}
  (1-\phi_{\alpha t}(z)) \sup_{\Omega\cap \mathbb D_t} h \ge \left[1-\log \frac4t \cdot\frac{\mathcal C_d(K_{\alpha t},\mathbb D)}{2\pi}\right]h(z),
  \end{equation}
  while the same inequality holds for $z\in \partial \Omega\cap \mathbb D_t$, because $\lim_{z\rightarrow \zeta}h=0$ for  n.e. $\zeta\in \partial \Omega\cap \mathbb D_t$. By the (generalized) maximum principle,  \eqref{eq:G_2} holds on $\Omega\cap \mathbb D_t$. On the other hand, since for $|z|= \alpha t$ we have
  \begin{equation}\label{eq:G_3}
  \inf_{\partial K_{\alpha t}}(-g_{\mathbb D}(\cdot,z)) \ge \log 1/2+  \inf_{\partial K_{\alpha t}}(-\log|\cdot-z|) \ge \log 1/2-\log (2\alpha  t) =  \log \frac1{4\alpha  t}, 
  \end{equation}
  it follows from \eqref{eq:C_4''} that 
  \begin{equation}\label{eq:G_4}
  \phi_{\alpha t}(z)\ge \log \frac1{4\alpha t}\cdot \frac{\mathcal C_d(K_{\alpha t},\mathbb D)}{2\pi}\ \ \ \  \text{for}\ \ |z|=\alpha t.
    \end{equation}
   Substituting \eqref{eq:G_4} into \eqref{eq:G_2},   we have
    \begin{eqnarray}\label{eq:G_5}
    h(z) & \le & \sup_{\Omega\cap \mathbb D_t} h \cdot \frac{1-\log \frac1{4\alpha t}\cdot\frac{\mathcal C_d(K_{\alpha t},\mathbb D)}{2\pi}}{1-\log \frac4t \cdot\frac{\mathcal C_d(K_{\alpha t},\mathbb D)}{2\pi}}\nonumber\\
        &\le & \sup_{\Omega\cap \mathbb D_t} h \left(1+\frac{\log (16\alpha)\cdot\frac{\mathcal C_d(K_{\alpha t},\mathbb D)}{2\pi}}{1-\log \frac4t \cdot\frac{\mathcal C_d(K_{\alpha t},\mathbb D)}{2\pi}}\right)
            \end{eqnarray}
            for $z\in \Omega\cap \mathbb D_{\alpha t}$. Set $M(t):=\sup_{\Omega\cap \mathbb D_t} h$. It follows from \eqref{eq:G_5} and \eqref{eq:C_8} that 
      \begin{eqnarray*}
      \frac{\log M(t)}t-\frac{\log M(\alpha t)}t & \ge & \log\frac1{16\alpha}\left({t\log \frac{t}{4 \mathcal C_g(K_{\alpha t},\mathbb D)}} \right)^{-1}\\
      & \ge & \log\frac1{16\alpha} \left({t\log \frac{t}{2 \mathcal C_l(K_{\alpha t})}} \right)^{-1}\ \ \ \ (\text{by\ } \eqref{eq:C_9}).
               \end{eqnarray*}  
               Integration from $r/\alpha$ to $r_0$ gives    
               \begin{eqnarray*}
            && \log\frac1{16\alpha}  \int_{r/\alpha}^{r_0} \left[{t\log \frac{t}{2 \mathcal C_l(K_{\alpha t})}} \right]^{-1} dt\\
              & \le &   \int_{r/\alpha}^{r_0}   \frac{\log M(t)}t dt- \int_{r/\alpha}^{r_0}  \frac{\log M(\alpha t)}t dt\\
               & = &  \int_{r/\alpha}^{r_0}   \frac{\log M(t)}t dt- \int_{r}^{\alpha r_0}  \frac{\log M( t)}t dt\\
               & \le &   \int_{\alpha r_0}^{r_0}   \frac{\log M(t)}t dt  -\int_r^{r/\alpha}  \frac{\log M(t)}t dt\\
               &\le & (\log M(r_0)-\log M(r)) \log 1/\alpha,
                             \end{eqnarray*}
            because $M(t)$ is nondecreasing.  Thus \eqref{eq:G_0} holds because $M(r_0)\le 1$.
            \end{proof}  
                                   
            \begin{theorem}\label{th:basic}
            Fix a compact set $E$ in $\Omega$ with $\mathcal C_l(E)>0$. Let $\phi_{E}$ be the capacity potential of $E$ relative to $\Omega$.  Set $d=d(E,\partial \Omega)$. Then for all $0<\alpha<1/16$
             \begin{equation}\label{eq:G_6}
    \sup_{\Omega\cap \mathbb D_r} \phi_{E} \le \exp\left[-\frac{\log 1/(16\alpha)}{\log 1/\alpha} \int_{r}^{\alpha d} \left({t\log \frac{t/\alpha}{2 \mathcal C_l(K_{t})}} \right)^{-1} dt  \right].
    \end{equation}     
           \end{theorem}
           
           \begin{proof}
           The solution of the (generalized) Dirichlet problem gives $\lim_{z\rightarrow \zeta}\phi_E(z)=0$ for n.e. $\zeta\in \partial \Omega$, and the (generalized) maximum principle gives $0\le \phi_E\le 1$.  Thus Lemma \ref{lm:basic} applies.
           \end{proof}

    \section{Growth of the Green function}
           
          \begin{proof}[Proof of Theorem \ref{th:CT}]
           Since $ \mathcal D_W(\varepsilon,\lambda)>0$,  there exist $c>0$ and $n_0\in \mathbb Z^+$ such that 
           $$
           |\mathcal N_a^n(\varepsilon,\lambda)|\ge c n,\ \ \ \forall\,n\ge n_0\ \text{and}\ \ a\in \partial \Omega.
                      $$
                       Since $\mathcal N_a^n(\varepsilon,\lambda)$ is decreasing in $\varepsilon$, we may assume that $\varepsilon$ is as small as we want. 
           Note that for $n\gg N\gg 1$
        \begin{eqnarray}\label{eq:G_8'}
        \int_{\lambda^n}^{\lambda^N} \left({t\log \frac{t/\alpha}{2 \mathcal C_l(K_{t}(a))}} \right)^{-1} dt & \ge & \sum_{k\in \mathcal N_a^n(\varepsilon,\lambda)\backslash \mathcal N_a^N(\varepsilon,\lambda)}
        \int_{\lambda^{k}}^{\lambda^{k-1}} \left({t\log \frac{t/\alpha}{2 \mathcal C_l(K_{t}(a))}} \right)^{-1} dt\nonumber\\
        & \ge &  \sum_{k\in \mathcal N_a^n(\varepsilon,\lambda)\backslash \mathcal N_a^N(\varepsilon,\lambda)} \left(\log\frac1{2\lambda\varepsilon\alpha}\right)^{-1} \int_{\lambda^{k}}^{\lambda^{k-1}}  \frac{dt}t   \nonumber \\
                 & = &
         \log1/\lambda \cdot \left(\log\frac1{2\lambda\varepsilon\alpha}\right)^{-1} |\mathcal N_a^n(\varepsilon,\lambda)\backslash \mathcal N_a^N(\varepsilon,\lambda)|\nonumber\\
         &\ge &  \log1/\lambda \cdot \left(\log\frac1{2\lambda\varepsilon\alpha}\right)^{-1}\cdot\frac{cn}2\nonumber\\
         & = &    \frac{c}2\cdot   \left(\log\frac1{2\lambda\varepsilon\alpha}\right)^{-1}\cdot \log 1/\lambda^n.  
         \end{eqnarray}  
         Since for every $z$ there exists $n\in \mathbb Z^+$ such that $\lambda^n\le |z-a|\le \lambda^{n-1}$,  it follows from \eqref{eq:G_6} and \eqref{eq:G_8'} that 
         $$
         \phi_E(z)\lesssim |z-a|^\beta
         $$
         for suitable  constant $\beta>0$ which is independent of $a$.  Since $-g_\Omega(z,z_0)\asymp \phi_E(z)$ for all $z$ sufficiently close to $\partial \Omega$,  we conclude that  \eqref{eq:CT} holds.
         \end{proof}  
         
          \begin{proof}[Proof of Theorem \ref{th:logGreen}]
        $(1)$   Since $ \mathcal D_W(\varepsilon,\lambda,\gamma)>0$,  there exist $c>0$ and $n_0\in \mathbb Z^+$ such that 
           $$
          \sum_{k\in \mathcal N_a^n(\varepsilon,\lambda,\gamma)} k^{-1}  \ge c \log n,\ \ \ \forall\,n\ge n_0\ \ \text{and}\ \ a\in \partial \Omega.
                      $$
           Note that for $n\gg N\gg 1$,
        \begin{eqnarray*}
        \int_{\lambda^n}^{\lambda^N} \left({t\log \frac{t/\alpha}{2\mathcal C_l(K_{t}(a))}} \right)^{-1} dt 
        & \ge & \sum_{k\in \mathcal N_a^n(\varepsilon,\lambda,\gamma)\backslash \mathcal N_a^N(\varepsilon,\lambda,\gamma)}
        \int_{\lambda^{k}}^{\lambda^{k-1}} \left({t\log \frac{t/\alpha}{2\mathcal C_l(K_{t}(a))}} \right)^{-1} dt\\
        & \ge &  \sum_{k\in \mathcal N_a^n(\varepsilon,\lambda,\gamma)\backslash \mathcal N_a^N(\varepsilon,\lambda,\gamma)} \left(\log\frac{\lambda^{(1-\gamma)k-1}}{2 \varepsilon\alpha}\right)^{-1} \int_{\lambda^{k}}^{\lambda^{k-1}}  \frac{dt}t    \\
                 & \gtrsim & \sum_{k\in \mathcal N_a^n(\varepsilon,\lambda,\gamma)\backslash \mathcal N_a^N(\varepsilon,\lambda,\gamma)} k^{-1}\\
         & \gtrsim &   \log n 
         \end{eqnarray*}  
         where the implicit constants are independent of $a$.
      This combined with   \eqref{eq:G_6}  gives 
      $$
      \phi_E(z)\lesssim (-\log |z-a|)^{-\beta}
      $$
      for some constant $\beta>0$ independent of $a$, which in turn implies 
     \eqref{eq:LG_1}.
        
$(2)$   
            By \eqref{eq:G_6} we have for every $a\in \partial \Omega$
            \begin{eqnarray*}
    \sup_{\Omega\cap \mathbb D_r(a)} \phi_{E} & \le & \exp\left[-\frac{\log 1/(16\alpha)}{\log 1/\alpha} \int_{r}^{\alpha d} \left({t\log \frac{t/\alpha}{2\mathcal C_l(K_{t}(a))}} \right)^{-1} dt  \right]\\
    & \le & \exp\left[-\frac{\log 1/(16\alpha)}{\log 1/\alpha} \int_{r}^{\alpha d} \frac{dt} {t((\gamma-1)\log 1/t+ \log 1/(2\alpha \varepsilon) )}   \right]\\
    & \le & {\rm const}_\tau\, (-\log r)^\tau
        \end{eqnarray*}
        provided $\alpha$ sufficiently small, from which the assertion follows. 
                   \end{proof}
                   
                   \section{Lower bounds of the Bergman kernel}
\subsection{Proof of Theorem \ref{th:CapBerg}}  
           \begin{lemma}[cf. \cite{Grigoryan}]
           Let $\Omega$ be a bounded domain in $\mathbb C$ and $U$ a relatively compact open set in $\Omega$. For every $w\in U$ we have
         \begin{equation}\label{eq:G_12}
         \min_{\partial U} (-g_\Omega(\cdot,w))\le \frac{2\pi}{\mathcal C_d\left(\overline{U},\Omega\right)} \le \max_{\partial U} (-g_\Omega(\cdot,w)).
         \end{equation}   
           \end{lemma}
           
           \begin{proof}
          Since $g_\Omega(\cdot,w)$ is harmonic on $\Omega\backslash \overline{U}$ and vanishes n.e on $\partial \Omega$, it follows from  the maximum principle that
           $$
           \sup_{\Omega\backslash U} (-g_\Omega(\cdot,w)) =  \max_{\partial U} (-g_\Omega(\cdot,w))\ \ \text{and}\ \    \inf_{ \overline{U}} (-g_\Omega(\cdot,w)) =  \min_{\partial U} (-g_\Omega(\cdot,w)).
                   $$
                   Then we have
                   $$
               \left\{-g_\Omega(\cdot,w)\ge \max_{\partial U} (-g_\Omega(\cdot,w))\right\} \subset   \overline{U}\subset \left\{-g_\Omega(\cdot,w)\ge \min_{\partial U} (-g_\Omega(\cdot,w))\right\}.
                   $$
                   Set $F_c=\{-g_\Omega(\cdot,w)\ge c\}$. It suffices to show
                   $$
                   \mathcal C_d(F_c,\Omega) =2\pi/c.
                   $$ 
                   Indeed, the function $\phi_c:=-c^{-1} g_\Omega(\cdot,w)$ is the capacity potential of $F_c$ relative to $\Omega$. Thus we have
                   $$
                 \mathcal  C_d(F_c,\Omega) = -\text{flux}_{\partial F_c}\, \phi_c= -\text{flux}_{\partial \Omega}\, \phi_c = c^{-1} \text{flux}_{\partial \Omega}\, g_\Omega(\cdot,w)=2\pi/c.
                   $$
                   where the second and last equalities follow from Green's formula.
                  \end{proof}
                  
                  \begin{lemma}\label{lm:Green_2}
               Let $\Omega$ be a bounded domain in $\mathbb C$ with $0\in \partial \Omega$. Let $\beta> \alpha>0$. Suppose $\mathcal C_l(K_r)\ge \varepsilon r$ for some $\varepsilon,r>0$.  There exists a positive number $c$ depending only on $\alpha,\beta,\varepsilon$  such that for every point $w$ with $|w|=\beta r$ and               
                 $\mathbb D_{2\alpha r}(w)\subset \Omega$ we have
               \begin{equation}\label{eq:G_13}
               \{ g_\Omega(\cdot,w)\le -c\}\subset \mathbb D_{\alpha r}(w).
               \end{equation} 
                  \end{lemma}
                  
                  \begin{proof}
                  By \eqref{eq:G_12} and Harnack's inequality it suffices to show
                  \begin{equation}\label{eq:G_14}
                \mathcal   C_d\left(\overline{\mathbb D_{\alpha r}(w)},\Omega\right)\ge c'
                  \end{equation}
                  for some positive constant $c'$ depending only on $\alpha,\beta,\varepsilon$. By the definition  we see that 
                  \begin{eqnarray}\label{eq:G_15}
               \mathcal   C_d\left(\overline{\mathbb D_{\alpha r}(w)},\Omega\right) & = & \mathcal C_d\left(\Omega^c,\mathbb C_\infty-\overline{\mathbb D_{\alpha r}(w)}\right)\nonumber\\
                  & \ge & \mathcal C_d\left(K_r,\mathbb C_\infty-\overline{\mathbb D_{\alpha r}(w)}\right),
                                    \end{eqnarray}
       where $\mathbb C_\infty$ denotes the Riemann sphere.  
                                    Let us consider the conformal map
                                    $$
                                    T: \mathbb C_\infty-\overline{\mathbb D_{\alpha r}(w)}\rightarrow \mathbb D,\ \ \ z\mapsto 
                                    \frac{\alpha r}{z-w}.
                                                                        $$
                              Since the Dirichlet energy is invariant under conformal maps, it follows that
                              \begin{equation}\label{eq:G_16}
                         \mathcal    C_d\left(K_r,\mathbb C_\infty-\overline{\mathbb D_{\alpha r}(w)}\right) = \mathcal C_d(T(K_r),\mathbb D).
                              \end{equation}
                           Since 
                           $$
                           K_r\subset  \mathbb D_{(1+\beta) r}(w)-{\mathbb D_{2\alpha r}(w)},
                           $$
                            we have
                           $$
                           T(K_r)\subset \overline{\mathbb D}_{1/2}-\mathbb D_{\alpha/(1+\beta)},
                           $$
                            so that
                           \begin{eqnarray}\label{eq:G_17}
                       \mathcal    C_d(T(K_r),\mathbb D) =-\frac{2\pi}{\log \mathcal C_g(T(K_r),\mathbb D)} \ge \frac{2\pi}{\log 2-\log \mathcal C_l(T(K_r))}.                         \end{eqnarray}  
                           Since 
                           $$
                           |T^{-1}(z_1)-T^{-1}(z_2)|=\frac{\alpha r}{|z_1z_2|}\cdot |z_1-z_2|\le \frac{(1+\beta)^2r}{\alpha }\cdot |z_1-z_2|
                           $$ 
                           for all $z_1,z_2\in T(K_r)$, we infer from Theorem 5.3.1 in \cite{Ransford} that
                           $$
                           \mathcal C_l(T(K_r))\ge   \frac{\alpha}{(1+\beta)^2r}\cdot \mathcal C_l(K_r) \ge \frac{\alpha \varepsilon}{(1+\beta)^2}.
                                                  $$  
                                                  This combined with $\eqref{eq:G_15} \sim \eqref{eq:G_17}$ gives \eqref{eq:G_14}.                                                            
                  \end{proof}
                  
                  \begin{proof}[Proof of Theorem \ref{th:CapBerg}]
          Take $z^\ast\in \mathcal L_z\cap B(z,\alpha\delta_\Omega(z)) \cap \partial \Omega$.  Since 
          $$
          B(z,\alpha\delta_\Omega(z))\subset B(z^\ast,2\alpha \delta_\Omega(z)),
          $$
           we have
     $$
     \mathcal C_l\left( \mathcal L_z \cap B(z^\ast,2\alpha \delta_\Omega(z))\cap \Omega^c\right)\ge \frac{\varepsilon}{2\alpha}\cdot (2\alpha\delta_\Omega(z)). 
     $$
     By Lemma \ref{lm:Green_2}, there exists a constant  $c=c(\alpha,\varepsilon)>0$ such that
     $$
     \{g_{\mathcal L_z\cap \Omega}(\cdot,z)\le -c\}\subset \mathcal L_z\cap B(z,\delta_\Omega(z)/2).
     $$
     Hence it follows from \eqref{eq:BG_1} that
     $$
     K_{\mathcal L_z\cap \Omega}(z)\gtrsim |\{g_{\mathcal L_z\cap \Omega}(\cdot,z)\le -c\} |^{-1}\gtrsim \delta_\Omega(z)^{-2},
     $$
     where the implicit constants depend only on $\alpha,\varepsilon$.   This together with the   Ohsawa-Takegoshi extension theorem  yield  \eqref{eq:CB_2}.        
                  \end{proof}
    \subsection{Proof of Theorem \ref{th:KernelEstimate}}              
                  \begin{proof}[Proof of\/ $(b)$]
                   Take $z^\ast\in \partial \Omega$ such that $|z-z^\ast| = \delta_\Omega(z)$.  Since  $$
                   B(z^\ast,\delta_\Omega(z))\subset B(z,2\delta_\Omega(z)),
                   $$
                    it follows from \eqref{eq:CB_3} that
     $$
     |B(z,2\delta_\Omega(z))\cap \Omega^c|\ge C_n\varepsilon \delta_\Omega(z)^{2n},
     $$
     where $C_n$ stands for a constant depending only on  $n$.  We claim that there exists a real line  $l_z\ni z$,  such that
     \begin{equation}\label{eq:kl_3'}
     | l_z\cap B(z,2\delta_\Omega(z))\cap \Omega^c|_1 \ge \varepsilon' \delta_\Omega(z),
     \end{equation}
     where $\varepsilon'=\varepsilon'(n,\varepsilon)>0$,  $|\cdot|_1$ stands for the $1-$dimensional   Lebesgue measure.  Indeed,  if we denote $E:=B(z,2\delta_\Omega(z))\cap \Omega^c$ and $E_\zeta:=E\cap l_{z,\zeta}$,  where $\zeta\in \mathbb C^n$,  $l_{z,\zeta}=\{z+t\zeta:t\in \mathbb R\}$,   then
     \begin{eqnarray}\label{eq:Vol}
     |E| & = & \int_{\mathbb \zeta\in \mathbb S^{2n-1}} \int_{r: z+ r\zeta\in E} r^{2n-1}dr d\sigma(\zeta) \\
     & \le & 2^{2n-1} \delta_\Omega(z)^{2n-1} \int_{\mathbb \zeta\in \mathbb S^{2n-1}} |E_\zeta|_1 d\sigma(\zeta),\nonumber
     \end{eqnarray}
     where $\mathbb S^{2n-1}$ is the unit sphere in $\mathbb C^n$ and $d\sigma$ is the surface element.  Thus
     $$
     \int_{\mathbb \zeta\in \mathbb S^{2n-1}} |E_\zeta|_1 d\sigma(\zeta)\ge C_n 2^{1-2n}\varepsilon \delta_\Omega(z),
     $$
     so that there exists at least one point  $\zeta_0\in \mathbb S^{2n-1}$,  such that
     $$
     |E_{\zeta_0}|_1\ge \frac{C_n\varepsilon}{2^{2n}|\mathbb S^{2n-1}|} \cdot \delta_\Omega(z).
     $$
     Take $l_z=l_{z,\zeta_0}$ and $\varepsilon'= \frac{C_n\varepsilon}{2^{2n}|\mathbb S^{2n-1}|}$,  we get \eqref{eq:kl_3'}.  
     
     Let $\mathcal L_z$ be the complex line determined by $l_z$.  Then
     $$
     \mathcal C_l(\mathcal L_z\cap B(z,2\delta_\Omega(z))\cap \Omega^c)\ge \frac{|E_{\zeta_0}|_1}4\ge \frac{\varepsilon'}4\cdot\delta_\Omega(z),
     $$
  where the first inequality follows from Theorem 5.3.2 in \cite{Ransford}.   Thus Theorem \ref{th:CapBerg} applies.  
                  \end{proof}
                  
                  \begin{proof}[Proof of\/ $(d)$]
            By Theorem \ref{th:CapBerg},  it suffices to verify that there exists a complex line  $\mathcal L_z\ni z$,  such that
\begin{equation}\label{eq:kl_4}   
\mathcal C_l(\mathcal L_z\cap  B(z,2\delta_\Omega(z))\cap \Omega^c) \ge \delta_\Omega(z)/4.
\end{equation}    
     This follows from an  argument of Pflug-Zwonek \cite{PZ05}:  Take     
      $c_0>0$,  so that $\{\rho<c_0\}\subset\subset U$,  where $\rho$ is defined on an open set $U\supset \overline{\Omega}$.  Take $z^\ast\in \partial \Omega\cap \partial B(z,\delta_\Omega(z))$ and let   $\mathcal L_z$ be the complex line determined by $z$ and $z^\ast$.  We claim that
     \begin{equation}\label{eq:S-hyper}
   \mathcal L_z\cap \partial B(z,r)\cap \Omega^c\neq \emptyset, \ \ \ \delta_\Omega(z)<r<d_0-\delta_\Omega(z),
     \end{equation}
     where $d_0=d(\overline{\Omega},\{\rho=c_0\})$.
     Indeed,  if $ \mathcal L_z\cap \partial B(z,r)\cap \Omega^c= \emptyset$,  i.e.,
     $$
     \mathcal L_z\cap \partial B(z,r)\subset \Omega.  
     $$
     then  the maximum principle and continuity of  $\rho$ yield 
     $$
     \rho(z^\ast)\le \max\{\rho(w):w\in \mathcal L_z\cap \partial B(z,r)\}<0 =\rho(z^\ast), 
     $$ 
     this is a contradiction!
     
     Since the circle projection $T:w\mapsto |w|$ satisfies $|T(w)-T(w')|\le |w-w'|$,  it follows from Theorem 5.3.1 and Theorem 5.3.2 in \cite{Ransford} that   \eqref{eq:S-hyper} gives \eqref{eq:kl_4}.       
                  \end{proof}

         \begin{proof}[Proof of\/ $(c)$]
    Given $0<\varepsilon\ll 1$,  define
    $$
    \Omega_\varepsilon:=\left\{z\in U: d(z,\overline{\Omega})< \varepsilon\right\}.
    $$
    Since $\overline{\Omega}^\circ =\Omega$,  it is not difficult to show that $\delta_{\Omega_\varepsilon}(z)\rightarrow \delta_\Omega(z)$  uniformly on $\overline{\Omega}$ as $\varepsilon\rightarrow 0$.   Fix a domain $V$ with $\overline{\Omega}\subset V\subset\subset U$.  
   Since $\overline{\Omega}$ is  $\mathcal O(U)-$convex,   there exist a finite number of holomorphic functions $f_1,\cdots,f_m$ on $U$,  where $m=m_\varepsilon$,  such that 
$$
\sup_{\overline{\Omega}}\max_{1\le j\le m} |f_j|<1\ \ \text{and}\ \  \inf_{\overline{V}\backslash \Omega_\varepsilon}  \max_{1\le j\le m} |f_j|\ge 1.
$$   
   Clearly,  $\rho:=\max_{1\le j\le m} |f_j|-1$ is a continuous psh function on $U$,  and we have 
   $$
   \overline{\Omega}\subset \{z\in V: \rho(z)<0\}=:\Omega'\subset \Omega_\varepsilon.
   $$
   Note that $\Omega$ is pseudoconvex since $\overline{\Omega}$ is the limit of a decreasing family of analytic polyhedrons.    
    
   Let $z\in \Omega$ be a fixed point which is sufficiently close to $\partial \Omega$.  Take $\varepsilon\ll 1$ such that $\delta_{\Omega_\varepsilon}(z)\le 3\delta_\Omega(z)/2$.  Then we have 
   $$
   \delta_{\Omega'}(z)\le \delta_{\Omega_\varepsilon}(z)\le 3\delta_\Omega(z)/2.
   $$
   Take $z^\ast\in \partial \Omega'\cap \partial B(z,\delta_{\Omega'}(z))$ and let   $\mathcal L_z$ be the complex line determined by $z$ and $z^\ast$.  From the proof of $(d)$,  we know that
    $$
   \mathcal L_z\cap \partial B(z,r)\cap (\mathbb C^n\setminus \Omega')\neq \emptyset, \ \ \ \delta_{\Omega'}(z)<r<r_0,
     $$
     where $r_0>0$ is suitable constant.  Since $\delta_{\Omega'}(z)\le \frac32\, \delta_\Omega(z)$,  we conclude  that
 $$
\mathcal C_l(\mathcal L_z\cap  B(z,2\delta_\Omega(z))\cap \Omega^c) \ge \mathcal C_l(\mathcal L_z\cap  B(z,2\delta_\Omega(z))\cap (\mathbb C^n\setminus \Omega')) \ge \frac14\cdot \frac{\delta_\Omega(z)}2
$$
in view of the circular projection.  
Thus Theorem \ref{th:CapBerg} applies.  
    \end{proof}

 \begin{proof}[Proof of\/ $(a)$] 
Let $z_0\in \partial \Omega$.  After a complex affine  transformation,  we may assume that $z_0=0$ and there exist a ball in $\mathbb C^{n-1}\times \mathbb R$ given by
$$
B'_{r_0}:=\left\{(z',\mathrm{Re}\,z_n)\in \mathbb C^{n-1} \times \mathbb R: |z'|^2+(\mathrm{Re}\,z_n)^2<r_0^2\right\},
$$ 
and a continuous real-valued function $g$ on $B'_{r_0}$ such that
$$
\Omega\cap \left\{B'_{r_0}\times (-r_0,r_0)\right\}=\left\{z\in B'_{r_0}\times (-r_0,r_0): \mathrm{Im}\,z_n < g(z',\mathrm{Re}\,z_n) \right\}.
$$
Given $0\le t\le \varepsilon_0\ll r_0$,  define
$$
\Omega_t:=\left\{z\in B'_{\frac{r_0}2+t}\times (-r_0/2-t,r_0): \mathrm{Im}\,z_n < g(z',\mathrm{Re}\,z_n) +t \right\}.
$$
Note that $\Omega_0=\Omega\cap \left\{B'_{\frac{r_0}2}\times (-r_0/2,r_0)\right\}$ and $\{\Omega_t\}_{t\in [0,\varepsilon_0]}$ forms an increasing $1-$parameter family of bounded pseudoconvex\/ {\it fat}\/  domains.  

By $(c)$ and the well-known localization principle of the Bergman kernel,  it suffices to show that the closure $\overline{\Omega}_0$ of $\Omega_0$ is $\mathcal O(\Omega_{\varepsilon_0})-$convex.  Given $w\in \Omega_{\varepsilon_0}\setminus \overline{\Omega}_0$,  take $0<t_1<\varepsilon_0$ such that $w\in \partial \Omega_{t_1}$.  
Take $\gamma>0$  and  $t_1<t_2<\varepsilon_0$ such that   $\min_{\overline{\Omega}_0}\delta_{\Omega_{t_1}}>\gamma$ and $\delta_{\Omega_{t_2}}(w)=\gamma/2$.   We are going to construct a holomorphic function $f$ on $\Omega_{t_2}$ satisfying $|f(w)|>1$ and $|f|<1$ on $\overline{\Omega}_0$.  The  argument is essentially standard,   but we include the detail  here for the sake of completeness.  Define $\psi:=-\log \delta_{\Omega_{t_2}}+\log \gamma$.  Clearly,  $\psi$ is a locally Lipschitz continuous psh exhaustion function on $\Omega_{t_2}$  satisfying
$$
\overline{\Omega}_0\subset \{\psi<0\}\ \ \text{and}\ \ \psi(w)=\log 2.
$$
Let $\kappa$ be a convex nondecreasing function on $\mathbb R$ with $\kappa=0$ on $(-\infty,0]$ and $\kappa',\kappa''>0$ on $(0,\infty)$. 
Take a smooth cut-off function $\chi$ on $\mathbb R$ such that $\chi=0$ on $[\frac14 \log 2,\frac13 \log 2]\cup [3\log 2,\infty)$ and $\chi=1$ on $(-\infty,0]\cup [\frac12 \log 2,2\log 2]$.
 Define $h=\frac12$ on $\{\psi\le \frac14 \log 2\}$ and $h=2$ on $\{\frac13 \log 2\le \psi \le 3\log 2\}$.  Then $\chi(\psi) h$ defines a Lipschitz continuous function with compact support in $\Omega_{t_2}$.  By H\"ormander's $L^2-$estimates (cf.  \cite{Hormander}),  there exists a solution of  $\bar{\partial}u=\bar{\partial}(\chi(\psi) h)$ such that
 \begin{eqnarray*}
 \int_{\Omega_{t_2}}  |u|^2 e^{-\kappa(\psi)-2n\log |z-w|} & \lesssim & \int_{\Omega_{t_2}}  |h|^2 |\chi'(\psi)|^2 |\bar{\partial}\psi|^2 e^{-\kappa(\psi)-2n\log |z-w|}\\
 & \lesssim & \int_{\{0<\psi<\frac12 \log 2\}\cup \{2\log 2<\psi<3\log 2\}} e^{-\kappa(\psi)} |\bar{\partial}\psi|^2,
 \end{eqnarray*}
 where the implicit constants are independent of the choice of $\kappa$,  which can be arbitrarily small provided that $\kappa$ increases on $(0,\infty)$ sufficiently rapidly.  Thus $f:=\chi(\psi)h-u$ is a holomorphic function on $\Omega_{t_2}$ which satisfies $f(w)=h(w)=2$ and since $f-h=-u$ is holomorphic on $\{\psi<0\}$,  it follows the mean-value inequality that
\begin{eqnarray*}
\sup_{\overline{\Omega}_0} |f-h|^2  & \lesssim & 
\int_{\{\psi<0\}} |u|^2 e^{-2n\log |z-w|}\\
& \lesssim &  \int_{\Omega_{t_2}}  |u|^2 e^{-\kappa(\psi)-2n\log |z-w|}, 
\end{eqnarray*} 
which can be arbitrarily small provided that $\kappa$ increases on $(0,\infty)$ sufficiently rapidly.  Thus $|f|<1$ on $\overline{\Omega}_0$.  

Finally,  since $\{\Omega_t\}_{t\in [0,\varepsilon_0]}$ forms an increasing $1-$parameter family of bounded pseudoconvex domains,  it follows from a classical result of Docquier-Grauert \cite{DG} that $(\Omega_{t_2},\Omega_{\varepsilon_0})$ forms a Runge pair,  i.e.,  every holomorphic function on $\Omega_{t_2}$ can be approximated locally uniformly on $\Omega_{t_2}$ by holomorphic functions on $\Omega_{\varepsilon_0}$.  From this we obtain a holomorphic function $\widetilde{f}$ on $\Omega_{\varepsilon_0}$ satisfying $|\widetilde{f}(w)|>1$ and $|\widetilde{f}|<1$ on $\overline{\Omega}_0$.  Thus $\overline{\Omega}_0$ is $\mathcal O(\Omega_{\varepsilon_0})-$convex.
\end{proof}

\subsection{Proof of Proposition \ref{prop:Runge}}

We first give an elementary  lemma as follows.  

\begin{lemma}\label{lm:sc}
If\/ $D\subsetneq \mathbb C$ is simply-connected,  then 
$$
K_D(z)\ge \frac1{16\pi}\cdot \delta_D(z)^{-2},\ \ \ z\in D.
$$ 
\end{lemma}

\begin{proof}
Given $z\in D$,  consider the Riemann mapping $F:D\rightarrow \mathbb D$ satisfying
$$
 F(z)=0,\,F'(z)>0.
$$
It is a standard consequence of the Kobe one-quarter theorem that  
 $\delta_D(z)\ge \frac1{4|F'(z)|}$ (cf.  \cite{CG},  Theorem 1.4).   Thus
$$
K_D(z)=K_{\mathbb D}(0) |F'(z)|^2\ge \frac1{16 \pi}\cdot \delta_D(z)^{-2}.
$$
\end{proof}

\begin{proof}[Proof of Proposition \ref{prop:Runge}]
Note that the Runge domain $\Omega$ enjoys the following geometric property: for every complex line $\mathcal L$,  all components of $\Omega\cap L$ are  simply-connected.  To see this,  let $D$ be a component of $\Omega\cap \mathcal L$.  Given  $f\in \mathcal O(D)$,  the holomorphic function on $\Omega\cap \mathcal L$ which equals  $f$ on $D$ and $0$ on other components  can be extended to some $\widetilde{f}\in \mathcal O(\Omega)$,  for $\Omega$ is pseudoconvex.  
As $\Omega$ is Runge,  there exists a sequence of complex polynomials $\{P_j\}$ converging locally uniformly to $\widetilde{f}$ on $\Omega$.  Thus the polynomials  $\left\{P_j|_{\mathcal L}\right\}$ converge locally uniformly on $D$ to $f$,  so that $D$ has to be simply-connected.  

Given $z\in \Omega$,  take $z^\ast\in \partial \Omega$ so that $|z-z^\ast|=\delta_\Omega(z)$.  Let $\mathcal L_z$ be the complex line determined by $z,z^\ast$.  Let $D_z$ be the component of $\Omega\cap \mathcal L_z$ which contains $z$.  Clearly,  we have 
\begin{equation}\label{eq:sc}
\delta_{D_z}(z)=\delta_\Omega(z).
\end{equation}
 Set $f=K_{D_z}(\cdot,z)/\sqrt{K_{D_z}(z)}$ on $D_z$ and $f=0$ on other components of $\Omega\cap \mathcal L_z$.  By the Ohsawa-Takegoshi extension theorem,  there is a holomorphic function $\widetilde{f}$ on $\Omega$ such that $\widetilde{f}|_{\Omega\cap \mathcal L_z}=f$ and 
$$
\int_\Omega |\widetilde{f}|^2 \lesssim \int_{\Omega\cap \mathcal L_z} |f|^2 \lesssim 1,
$$
where the implicit constants depend only on the diameter of $\Omega$.  Thus
$$
K_\Omega(z)\ge \frac{|\widetilde{f}(z)|^2}{\int_\Omega |\widetilde{f}|^2} \gtrsim K_{D_z}(z) \gtrsim \delta_{D_z}(z)^{-2}
$$
in view of Lemma \ref{lm:sc}.  This together with \eqref{eq:sc} conclude the proof.
\end{proof}

  \subsection{Proofs of Theorem \ref{th:CapBerg_2} and Corollary \ref{cor:VolBerg_2}}           
                  The following lemma is essentially implicit in \cite{Zwonek}.
                  
                  \begin{lemma}\label{lm:CapBerg_2}
                Let $\Omega$ be a bounded domain in $\mathbb C$ and $\alpha>1$.  Then for any $z\in \Omega$ with $\alpha \delta_\Omega(z)\le 1/2$,
                \begin{equation}\label{eq:Z_1}
                K_\Omega(z)\gtrsim (\alpha\delta_\Omega(z))^{-2}(-\log \mathcal C_l(B(z,\alpha\delta_\Omega(z))\cap \Omega^c))^{-1},
                \end{equation}
                where the implicit constant  depends only on  $\mathrm{diam\,}\Omega$.  
                  \end{lemma}
                  
                  \begin{proof}
            Given $z\in \Omega$,   set $E:=\overline{B(z,\alpha\delta_\Omega(z))}\cap \Omega^c$.  Without loss of generality,  we assume $\mathcal C_l(E)>0$.  Dividing $\mathbb C$ into three trisection angles with vertices at $z$,  one may divide $E$ into three parts $E_1,E_2,E_3$.  We may assume $\mathcal C_l(E_1)\ge \mathcal C_l(E_2)\ge \mathcal C_l(E_3)$.  Then
\begin{equation}\label{eq:Z_2}
\frac1{-\log \mathcal C_l(E)}\le \frac1{-\log \mathcal C_l(E_1)}+ \frac1{-\log \mathcal C_l(E_2)} + \frac1{-\log \mathcal C_l(E_3)}\le \frac3{-\log \mathcal C_l(E_1)}.
\end{equation}            
                Let $\mu_{E_1}$ denote the   equilibrium measure of $E_1$.   Consider the following function
               $$
               f_{E_1}(\zeta):=\int_{E_1} \frac{d\mu_{E_1}(w)}{\zeta-w},\ \ \ \zeta\in \mathbb C\setminus E_1.
               $$
            Then $f_{E_1}$ is holomorphic on $\mathbb C\setminus E_1\supset \Omega$ and satisfies    
            $$
            \int_\Omega |f_{E_1}|^2 \lesssim {-\log \mathcal C_l(E_1)}  \lesssim {-\log \mathcal C_l(E)},
            $$
            where the implicit constant depends only on $\mathrm{diam\,}\Omega$ (cf. \cite{Zwonek},  Lemma 3).   On the other hand,  since $|z-w|\le \alpha \delta_\Omega(z)$ for all $w\in E_1$ and $E_1$ is contained in an angle with vertex $z$ and of the apex angle $2\pi/3$,  we conclude that there exists a numerical constant $C_0>0$ such that 
            $$
           | f_{E_1}(z) | \ge C_0 \int_{E_1} \frac{d\mu_{E_1}(w)}{|z-w|}\ge \frac{C_0}{\alpha \delta_\Omega(z)},
            $$
            for $d\mu_{E_1}$ is a probability measure.   Thus
            $$
            K_\Omega(z)\ge \frac{|f_{E_1}(z)|^2}{\|f\|_{L^2(\Omega)}^2}\gtrsim (\alpha \delta_\Omega(z))^{-2}(-\log \mathcal C_l(E))^{-1}.
            $$
                  \end{proof}
                  
        \begin{proof}[Proof of Theorem \ref{th:CapBerg_2}]
    The conclusion follows directly from  the Ohsawa-Takegoshi extension theorem and Lemma \ref{lm:CapBerg_2}.      
                  \end{proof}    
                  
                  \begin{proof}[Proof of Corollary \ref{cor:VolBerg_2}]   
          The argument is a mimic of the proof of Theorem \ref{th:KernelEstimate}/$(b)$.  Given $z\in \Omega$,  take $z^\ast\in \partial \Omega$ such that $|z-z^\ast| = \delta_\Omega(z)$.  Since  $B(z^\ast,\delta_\Omega(z))\subset B(z,2\delta_\Omega(z))$,  it follows from \eqref{eq:CB5} that
     \begin{equation}\label{eq:Z_3}
     |B(z,2\delta_\Omega(z))\cap \Omega^c|\ge \varepsilon (C_n \delta_\Omega(z))^{2n\beta}.
     \end{equation}
     We claim that there exists a real line  $l_z\ni z$,  such that
     \begin{equation}\label{eq:Z_4}
     | l_z\cap B(z,2\delta_\Omega(z))\cap \Omega^c|_1 \ge \varepsilon' \delta_\Omega(z)^{2n\beta-2n+1}
     \end{equation}
    for suitable $\varepsilon'=\varepsilon'(n,\varepsilon)>0$.  Indeed,  if we denote $E:=B(z,2\delta_\Omega(z))\cap \Omega^c$ and $E_\zeta:=E\cap l_{z,\zeta}$,   then \eqref{eq:Vol} together with \eqref{eq:Z_3}  yield
     $$
     \int_{\mathbb \zeta\in \mathbb S^{2n-1}} |E_\zeta|_1 d\sigma(\zeta)\ge C_n^{2n\beta} 2^{1-2n}\varepsilon \delta_\Omega(z)^{2n\beta-2n+1},
     $$
     so that there exists at least one point  $\zeta_0\in \mathbb S^{2n-1}$,  such that
     $$
     |E_{\zeta_0}|_1\ge \frac{C_n^{2n\beta}\varepsilon}{2^{2n}|\mathbb S^{2n-1}|} \cdot \delta_\Omega(z)^{2n\beta-2n+1}.
     $$
     Take $l_z=l_{z,\zeta_0}$ and $\varepsilon'= \frac{C_n^{2n\beta}\varepsilon}{2^{2n}|\mathbb S^{2n-1}|}$,  we get \eqref{eq:Z_4}.  
     
     Let $\mathcal L_z$ be the complex line determined by $l_z$.  Then
     $$
     \mathcal C_l(\mathcal L_z\cap B(z,2\delta_\Omega(z))\cap \Omega^c)\ge \frac{|E_{\zeta_0}|_1}4\ge \frac{\varepsilon'}4\cdot\delta_\Omega(z)^{2n\beta-2n+1}.
     $$
   Thus Theorem \ref{th:CapBerg_2} applies.          
                  \end{proof}  
                  
         \section{Lower bounds of the Bergman distance}
         
            \begin{proposition}\label{prop:Green_2}
       Suppose\/ $ \mathcal D_S(\varepsilon,\lambda)>0$\/ for some $\varepsilon,\lambda$.   There exists  $c\gg 1$  such that for every $k\in \bigcap_{a\in \partial \Omega}\mathcal N_a^n(\varepsilon,\lambda)$ and  every $w$ with $\lambda^{k-1/3}\le \delta_\Omega(w)\le \lambda^{k-1/2}$,
         \begin{equation}\label{eq:G_18}
         \{ g_\Omega(\cdot,w)\le -c\}\subset \left\{ \lambda^k< \delta_\Omega<\lambda^{k-1}\right\}.  
                \end{equation}
                  \end{proposition} 
                  
                  \begin{proof}
                 Take $a(w)\in \partial \Omega$  such that $|w-a(w)|=\delta_\Omega(w)$. Note that 
                 $$
                 \mathcal C_l(K_{\lambda^k}(a(w)))\ge \varepsilon \lambda^k.
                 $$
                 Thus Lemma \ref{lm:Green_2} applies.
                 \end{proof}

         \begin{proposition}\label{prop:Green_1}
       If\/ $ \mathcal D_W(\varepsilon,\lambda,\gamma)>0$\/ for some $\varepsilon,\lambda,\gamma$, then there exists $c\gg 1$ such that
         \begin{equation}\label{eq:G_10}
         \{ g_\Omega(\cdot,w)\le -1\}\subset \left\{c^{-1}\phi_E(w)^{\frac{1+\beta}\beta}< \phi_E < c\,\phi_E(w)^{\frac{\beta}{1+\beta}}\right\}  
                \end{equation}
                where $\beta$ is given as \eqref{eq:LG_1}.              
                     \end{proposition}
          
           \begin{proof}
         We shall first adopt a trick from \cite{CG}. Consider two points $z,w$ with $|z-w|$, $\delta_\Omega(z)$ and $\delta_\Omega(w)$ are sufficiently small.   We want to show
           \begin{equation}\label{eq:G_11}
           |\phi_E(z)-\phi_E(w)|\le c_0 (-\log |z-w|)^{-\beta}
           \end{equation}
           for some numerical constant $c_0>0$. Without loss of generality, we assume $\delta_\Omega(w)\ge \delta_\Omega(z)$. If $|z-w|\ge \delta_\Omega(w)/2$, this follows directly from \eqref{eq:LG_1}. Since $\phi_E$ is a positive harmonic function on $\mathbb D_{\delta_\Omega(w)}(w)$, we see that if $|z-w|\le \delta_\Omega(w)/2$ then  
           \begin{eqnarray*}
           |\phi_E(z)-\phi_E(w)| & \le & \sup_{\mathbb D_{\delta_\Omega(w)/2}(w)} |\nabla \phi_E|\, |z-w|\\
            & \le & c_1 \delta(w)^{-1} (-\log \delta_\Omega(w))^{-\beta} |z-w|\ \  \ (\text{by\ } \eqref{eq:LG_1})\\
            & \le & c_1 (2|z-w|)^{-1} (-\log (2|z-w|))^{-\beta} |z-w|  \\
            & \le & c_0 (-\log|z-w|)^{-\beta}.
                                 \end{eqnarray*}
 The remaining argument is standard.  Let $R$ be the diameter of $\Omega$. By (\ref{eq:G_11}) we conclude that if $\phi_E(z)=\phi_E(w)/2$ then      
 $$
 \log \frac{|z-w|}R \ge -\left(\frac{2c_0}{\phi_E(w)}\right)^{1/\beta}-\log R\ge -c_2 \phi_E(w)^{-1/\beta}
 $$    
      Since $-\phi_E$ is subharmonic on $\Omega$,  it follows that
$$
\psi(z):=\left\{
\begin{array}{ll}
 \log\frac{|z-w|}{R} & \text{ if\ } \phi_E(z)\ge \phi_E(w)/2\\
 \max\left\{\log\frac{|z-w|}{R},- 2c_2 \phi_E(w)^{-1-1/\beta} \phi_E(z)\right\} & \text{otherwise}.
\end{array}
\right.
$$
is a well-defined negative subharmonic function on $\Omega$ with a logarithmic pole at $w$, and if $\phi_E(z)\le \phi_E(w)/2$ then we have
$$
g_\Omega(z,w)\ge \psi(z)\ge - 2c_2 \phi_E(w)^{-1-1/\beta} \phi_E(z),
$$
so that
$$
\{g_\Omega(\cdot,w)\le -1\}\cap \{\phi_E\le \phi_E(w)/2\}\subset \left\{\phi_E\ge  (2c_2)^{-1}\phi_E(w)^{1+1/\beta}\right\}.
$$
 Since $\{\phi_E\ge  \phi_E (w)/2\} \subset \{\phi_E>  c^{-1}\phi_E(w)^{1+1/\beta}\}$ if $c\gg 1$, we have
 $$
 \{g_\Omega(\cdot,w)\le -1\} \subset \left\{\phi_E >  c^{-1}\phi_E(w)^{1+1/\beta}\right\}. 
 $$
 By the symmetry of $g_\Omega$, we immediately get
 $$
  \{g_\Omega(\cdot,w)\le -1\} \subset \left\{\phi_E <  (c\phi_E(w))^{\frac{\beta}{1+\beta}}\right\}.
   $$ 
           \end{proof}

             \begin{proof}[Proof of Theorem \ref{th:LowDist}]
$(1)$ Let $c$ be as Proposition \ref{prop:Green_2}. Let $z$ be  sufficiently close to $\partial \Omega$. Take $n\in \mathbb Z^+$ such that $\lambda^n\le \delta_\Omega(z)\le \lambda^{n-1}$. Write 
$$
\bigcap_{a\in \partial \Omega}\mathcal N_a^n(\varepsilon,\lambda)=\{k_1<k_2<\cdots<k_{m_n}\}.
$$
We may choose  a Bergman geodesic  jointing $z_0$ to $z$,  and a finite number of points  on this geodesic with the following order
  $$
  z_0 \rightarrow z_{k_1}\rightarrow z_{k_2}\rightarrow \cdots  \rightarrow z,
  $$
  such that
  $$
  \lambda^{k_j-1/3}\le \delta_\Omega(z_{k_j})\le \lambda^{k_j-1/2}.
    $$
  By Proposition \ref{prop:Green_2} we have 
  $$
  \{g_\Omega(\cdot,z_{k_j})\le -c\}\cap \{g_\Omega(\cdot,z_{k_{j+1}})\le -c\}=\emptyset
  $$
   so that  $d_B(z_{k_j},z_{k_{j+1}})\ge c_1>0$ for all $j$, in view of \eqref{eq:BG_2}.
  Since $m_n\gtrsim n$, we have
      $$
  d_B(z_0,z)  \ge  \sum_j d_B(z_{k_j},z_{k_{j+1}})\gtrsim n \gtrsim  |\log \delta_\Omega(z)|.
    $$
    
    $(2)$ The assertion follows directly from Theorem \ref{th:CT} and  Corollary 1.8 in \cite{ChenH-index}.
    
   $(3)$      Let $c$ be as Proposition \ref{prop:Green_1}.  Let $z$ be  sufficiently close to $\partial \Omega$. We may choose  a Bergman geodesic  jointing $z_0$ to $z$,  and a finite number of points $\{z_k\}_{k=1}^m$ on this geodesic with the following order
  $$
  z_0 \rightarrow z_1\rightarrow z_2\rightarrow \cdots \rightarrow z_m \rightarrow z,
  $$
  where
  $$
  c\,\phi_E(z_{k+1})^{\frac{\beta}{1+\beta}}=c^{-1}\phi_E(z_{k})^{\frac{1+\beta}\beta} 
   $$
  and
  $$
  c^{-1} \phi_E(z_m)^{\frac{1+\beta}\beta}\le  \phi_E(z)\le  c\phi_E(z_{m})^{\frac{\beta}{1+\beta}}.
     $$
  By Proposition \ref{prop:Green_1} we have 
  $$
  \{g_\Omega(\cdot,z_k)\le -1\}\cap \{g_\Omega(\cdot,z_{k+1})\le -1\}=\emptyset
  $$
   so that  $d_B(z_k,z_{k+1})\ge c_1>0$ for all $k$.
  
   Note that
   \begin{eqnarray*}
   \log \phi_E(z_0) & = & \left(\frac{\beta}{1+\beta}\right)^2 \log \phi_E(z_1)+ \frac{\beta}{1+\beta}\, \log c^2=\cdots \\
   & = & \left(\frac{\beta}{1+\beta}\right)^{2m} \log \phi_E(z_m) + \frac{\beta}{1+\beta}\,\frac{1-\left(\frac{\beta}{1+\beta}\right)^{2m}}{1-\left(\frac{\beta}{1+\beta}\right)^2}\, \log c^2. 
      \end{eqnarray*}
      Thus we have
      $$
      m\asymp \log |\log \phi_E(z_m)|\asymp \log |\log \phi_E(z)| \gtrsim \log\log |\log \delta_\Omega(z)|,
      $$
      so that 
    $$
  d_B(z_0,z)  \ge  \sum_{k=1}^{m-1} d_B(z_k,z_{k+1})\ge c_1(m-1) \gtrsim \log\log |\log \delta_\Omega(z)|.
    $$
 \end{proof}
           
   \section{Proof of Theorem \ref{th:HM}}
   
   Recall the following basic properties of quasiconformal (q.c.) mappings  (cf. \cite{Ahlfors}) and holomorphic motions (cf.  cf.  \cite{MSS},  \cite{BersRoyden}, \cite{Slodkowski},  \cite{ST}).  
   
   \begin{enumerate}
  \item (Circular Distortion Theorem). If $f$ is a $K-$q.c.  mapping of $\mathbb C$ fixing $\infty$, then 
 $$
|f(z_2)-f(z_0)|\le C_K\, |f(z_1)-f(z_0)|,\ \ \ \forall\,z_1,z_2\in \partial \mathbb D_r(z_0),\,\forall\,r>0.
$$
\item (H\"older continuity).  If $f$ is a $K-$q.c.  mapping of $\mathbb C$ fixing $\infty$, then
$$
C_K^{-1}|z-z'|^K \le |f(z)-f(z')|\le C_K|z-z'|^{1/K}.
$$
\item ($\lambda-$Lemma).    Every holomorphic motion $F:\mathbb D\times
E\rightarrow \mathbb C$  admits an extension to a holomorphic motion $F:\mathbb D\times \mathbb C\rightarrow
\mathbb C$ such that
  $F(\lambda,\cdot):\mathbb C\rightarrow \mathbb C$ is a
quasiconformal self-homeomorphism of dilatation not exceeding $
\frac{1+|\lambda|}{1-|\lambda|}$,  $\forall\,\lambda\in \mathbb D$.  
\end{enumerate}

   \begin{proposition}\label{prop:HM}
Given $0<r<1$,  there exists a constant $C>0$ such that   
   $$
   K_{ \Gamma_r(F)}((\lambda,w))\ge C \delta_{ \Gamma_r(F)}(\lambda,w)^{-2}.
   $$
   \end{proposition}
   
   \begin{proof}
   Clearly,  $\Gamma_r(F)$ is a fat domain.  On the other hand,  $\overline{\Gamma_r(F)}$ is polynomially convex (cf.  \cite{Slodkowski},  Proposition 3.3).  Thus Theorem \ref{th:KernelEstimate}/$(c)$ applies.  
   \end{proof}
   
  \begin{proof}[Proof of Theorem \ref{th:HM}]
   We shall borrow some techniques from \cite{CZ2}.  Given $z^\ast\in
\mathbb D$,  take $z_\ast\in \mathbb D^c$ with  $ z_\ast/ z^\ast\in \mathbb R^+$ and 
 $\frac{z^\ast+z_\ast}2\in \partial
\mathbb D$.  It follows the Circular Distortion Theorem that
\begin{eqnarray*}
\max_{|z-z^\ast|=1-|z^\ast|}|F(\lambda,z)-F(\lambda,z^\ast)| & \le &
C_r
\min_{|z-z^\ast|=1-|z^\ast|}|F(\lambda,z)-F(\lambda,z^\ast)|\\
\max_{|z-z_\ast|=1-|z^\ast|}|F(\lambda,z)-F(\lambda,z_\ast)| & \le &
C_r
\min_{|z-z_\ast|=1-|z^\ast|}|F(\lambda,z)-F(\lambda,z_\ast)|
\end{eqnarray*}
and
\begin{eqnarray*}
&& \max_{|z-\frac{z^\ast+z_\ast}2|=1-|z^\ast|}\left|F(\lambda,z)-F\left(\lambda,\frac{z^\ast+z_\ast}2\right)\right|\\
& \le & C_r
\min_{|z-\frac{z^\ast+z_\ast}2|=1-|z^\ast|}\left|F(\lambda,z)-F\left(\lambda,\frac{z^\ast+z_\ast}2\right)\right|
\end{eqnarray*}
hold for all $\lambda\in \overline{\mathbb D}_{r}$ and a (generic)
constant $C_r>0$.  It follows that if $C_r\gg 1$ then
\begin{equation}\label{eq:HM_5}
\bigcup_{\lambda\in \mathbb D_{r}}\left\{(\lambda,w)\in {\mathbb 
C}^2:|w-F(\lambda,z^\ast)|<C_r^{-1} \left|F\left(\lambda,z_\ast\right)-F(\lambda,z^\ast)\right|
\,\right\}\subset \Gamma_r(F)
\end{equation}
and
$$
h(\lambda,w):=\frac{F(\lambda,z^\ast)-F(\lambda,z_\ast)}{C_r[w-F(\lambda,z_\ast)]}
$$
gives a holomorphic mapping from $\Gamma_r(F)$ into $\mathbb D$.  Given
$(\lambda^\ast,F(\lambda^\ast,z^\ast))\in \Gamma_{r'}(F)$, 
define a holomorphic embedding as follows
\begin{eqnarray*}
\iota: && \Gamma_r(F)
\rightarrow  \mathbb D_r\times \mathbb D\\
&& (\lambda,w)  \mapsto (\lambda,w'):=
\left(\lambda,
\frac{h(\lambda,w)-h(\lambda,F(\lambda,z^\ast))}2\right).
\end{eqnarray*}
Since
\begin{eqnarray*}
h(\lambda,w)-h(\lambda,F(\lambda,z^\ast)) & = &
-C_r^{-1} \frac{w-F(\lambda,z^\ast)}{w-F(\lambda,z_\ast)},
\end{eqnarray*}
there exists a constant $s=s(r)<1$ such that
\begin{equation}\label{eq:inclusion}
\mathbb D_r\times \mathbb D\supset \iota (\Gamma_r(F))\supset \mathbb D_r \times \mathbb D_s.
\end{equation}
Set $w^\ast:=F(\lambda^\ast,z^\ast)$.  Note that
\begin{eqnarray*}
K_{\Gamma_r(F)}((\lambda^\ast,w^\ast)) & = & K_{\iota(\Gamma_r(F))}( \iota(\lambda^\ast,w^\ast) )  |  \partial w'/\partial w|^2\\
 & = & K_{\iota(\Gamma_r(F))}( (\lambda^\ast,0)) \, C_r^{-2} | F(\lambda^\ast,z^\ast)-F(\lambda^\ast,z_\ast)|^{-2} \\
 &\asymp & 
\left|F(\lambda^\ast,z^\ast)-F\left(\lambda^\ast,z_\ast\right)\right|^{-2}
\end{eqnarray*}
in view of \eqref{eq:inclusion},  where the implicit constants depend only $r,r'$.  This combined with Proposition \ref{prop:HM} gives
$$
\left|F(\lambda^\ast,z^\ast)-F\left(\lambda^\ast,z_\ast \right)\right|
 \lesssim  \delta_{ \Gamma_r(F)}(\lambda^\ast,w^\ast)
$$
whenever $\delta_{ \Gamma_r(F)}(\lambda^\ast,w^\ast)\le 1/2$.
Since
$$
\left|F(\lambda^\ast,z^\ast)-F\left(\lambda^\ast,z_\ast\right)\right|\ge \delta_{\lambda^\ast}(w^\ast),
$$
we conclude the proof.  
   \end{proof}

          \end{document}